\documentclass{amsart} 
\usepackage{amsfonts}

\theoremstyle{plain} 
\newtheorem{thm}{Theorem}

\newtheorem{lem}{Lemma}

\newtheorem*{thma}{Theorem A}
\newtheorem*{thmb}{Theorem B}

\theoremstyle{definition}

\theoremstyle{remark}

\DeclareMathOperator*{\stlim}{st-lim}

\newcommand{\C}{\mathbb{C}}
\newcommand{\R}{\mathbb{R}}

\newcommand{\e}{\varepsilon}
\newcommand{\loc}{\mathrm{loc}}
\newcommand{\imply}{\Rightarrow}

\def\refz#1{\/$\ref{#1}$}

\numberwithin{equation}{section}

\begin{document}

\title[Statistical extension of classical Tauberian theorems]{Statistical extension of classical Tauberian theorems in the case of logarithmic summability of locally integrable functions on $[1,\infty)$ } 

\author[F. M\'oricz {\protect\and} Z. N\'emeth]{Ferenc M\'oricz {\protect\and} Zolt\'an N\'emeth}
\address{Bolyai Institute, University of Szeged, Aradi v\'ertan\'uk tere 1, 6720 Szeged, Hungary}
\email{moricz@math.u-szeged.hu, znemeth@math.u-szeged.hu}

\begin{abstract} 
Let $s:[1,\infty) \to \C $ be a locally integrable function in Lebesgue's sense. The logarithmic (also called harmonic) mean of the function $s$ is defined by 
\[
\tau(t) := \frac 1{\log t} \int _1^t \frac {s(x)}{x} \,dx , \qquad t>1,
\]
where the logarithm is to base $e$. Besides the ordinary limit $\lim_{x\to \infty} s(x)$, we also use the notion of the so-called statistical limit of $s$ at $\infty$, in notation: $ \stlim_{x\to \infty} s(x)=\ell $, by which we mean that for every $\e>0$, 
\[
\lim_{b\to \infty} \frac 1b \Big | \Big \{ x\in(1,b): |s(x)-\ell| >\e \Big\} \Big| = 0 . 
\]
We also use the ordinary limit $\lim _{t\to\infty} \tau(t)$ as well as the statistical limit  $\stlim _{t\to\infty} \tau(t)$.

We will prove the following Tauberian theorem: Suppose that the real-valued function $s$ is slowly decreasing or the complex-valued $s$ is slowly oscillating. If the statistical limit   $\stlim _{t\to\infty} \tau(t) =\ell $ exists, then the ordinary limit $\lim _{x\to\infty} s(x) = \ell $ also exists.

\end{abstract}

\subjclass{
Primary 40C10, Secondary  40E05, 40G05. 
}

\keywords{
Statistical limit of a measurable function at $\infty$, logarithmic summability $(L,1)$ of a locally integrable function, slow decrease and slow oscillation with respect to summability $(L,1)$, nondiscrete version of a Vijayaraghavan lemma, Landau type one-sided and Hardy type two-sided Tauberian conditions.
}

\maketitle

\section{Introduction} 

We consider real- and complex-valued functions which are measurable (in Lebesgue's sense) on some interval $(a,\infty)$, where $a\ge 0$. We recall (see in \cite{kilenc}) that a function $s$ has \emph{statistical limit at} $\infty$ if there exists a number $\ell$ such that for every $\e>0$, 
\begin{equation} \label{1.1}
\lim_{b\to \infty} \frac {1}{b-a} \Big| \Big\{ x\in(a,b): |s(x)-\ell|>\e \Big\} \Big| = 0 , 
\end{equation} 
where by $|\{ \cdot \}| $ we denote the Lebesgue measure of the set indicated in $\{\cdot\}$. If this is the case, we write 
\begin{equation} \label{1.2}
\stlim_{x\to\infty} s(x) = \ell. 
\end{equation} 

Clearly, the statistical limit $\ell$ is uniquely determined if it exists. The particular choice of the left endpoint $a$ in the definition domain of the function $s$ is indifferent in \eqref{1.1}. Furthermore, if the ordinary limit 
\begin{equation} \label{1.3}
\lim_{x\to\infty} s(x) = \ell 
\end{equation} 
exists, then \eqref{1.2} also exists. But the converse implication is not true in general. 

We note that the notion of the statistical limit of a measurable function at $\infty$ is the nondiscrete version of the notion of the statistical convergence of a sequence of real or complex numbers, which was introduced by Fast \cite{egy}. 

If the function $s: [1,\infty) \to \C $ is integrable in Lebesgue's sense in every bounded interval $(1,t)$, $t>1$, in symbols: $ f\in L^1_\loc [1,\infty)$, then its logarithmic (also called harmonic) mean $\tau(t)$ of order $1$ is defined by 
\begin{equation} \label{1.4}
\tau(t) := \frac 1 {\log t} \int _1 ^ t \frac {s(x)}{x} \,dx , \qquad t>1,
\end{equation} 
where the logarithm is to the natural base $e$. The function $s$ is said to be \emph {logarithmic summable at} $\infty$, or biefly \emph{summable} $(L,1)$, if the finite limit 
\begin{equation} \label{1.5}
\lim_{t\to\infty} \tau(t) = \ell
\end{equation} 
exists. It is easy to check that if the ordinary limit \eqref{1.3} exists, then \eqref{1.5} also exists with the same $\ell$. 

The converse implication $ \eqref{1.5} \imply \eqref{1.3}$ is usually false. However, if we subject the function $s$ to an appropriate additional condition, then the implication $ \eqref{1.5} \imply \eqref{1.3}$ does hold. Such conditions are called `Tauberian' ones, and the theorems involving such conditions are also called `Tauberian' ones; after A.~Tauber \cite{tizenket}, who first proved one of the simplest of them. 

We recall (see in \cite{ujtiz}), that a function $s:[1,\infty) \to \R $ is said to be \emph{slowly decreasing} with respect to logarithmic summability, or briefly: summability $(L,1)$, if 
\begin{equation} \label{1.6}
\lim _{\lambda \to 1^+} \liminf _{x\to \infty} \inf _{\log x < \log t \le \lambda \log x } (s(t) - s(x)) \ge 0. 
\end{equation} 

Since the auxiliary function 
\[
a(\lambda) := \liminf _{x\to \infty} \inf _{\log x < \log t \le \lambda \log x}  (s(t) - s(x)),  \qquad \lambda>1, 
\]
is clearly decreasing in $\lambda$ on the interval $(1,\infty)$, the term `$  \lim_ {\lambda \to 1^+} $' in \eqref{1.6} can be equivalently replaced by `$\sup _{\lambda >1} $'. 

We observe that the conditions 
\[
\log x < \log t \le \lambda\log x \qquad \text{and} \qquad x< t \le x^\lambda, \qquad x>1, 
\]
are equivalent. In  the sequel, we will use the second one of them. It is easy to check that condition \eqref{1.6} is satisfied if and only if for every $\e>0$ there exist $x_0=x_0(\e)>1 $ and $ \lambda =\lambda(\e) >1$, the latter one is as close to 1 as we  want, such that 
\begin{equation} \label{1.7}
s(t) - s(x) \ge -\e \qquad \text{whenever} \quad x_0\le x < t \le x^\lambda . 
\end{equation}

Historically, the term `slow decrease' was introduced by Schmidt \cite{tizenegy} (see also in \cite[p.~124]{negy}), in the case of the summability $(C,1)$ of sequences of real numbers. 

We note that the symmetric counterpart of the notion of slow decrease is the folloving one: a real-valued function $s$ is said to be \emph {slowly increasing} with respect to summability $(L,1)$ if 
\[
\lim _{\lambda \to 1^+} \limsup _{x\to \infty} \sup _{x < t \le x^ \lambda } (s(t) - s(x)) \le 0. 
\]
Clearly, a function $s$ is slowly increasing if and only if the function $(-s)$ is slowly decreasing. 

We recall (see in \cite{ujtiz}), that a function $s: [1,\infty) \to \C $ is said to be \emph{slowly oscillating} with respect to summability $(L,1)$ if
\begin{equation} \label{1.8}
\lim _{\lambda \to 1^+} \limsup _{x\to \infty} \sup _{x < t \le x ^ \lambda} | s(t) - s(x) | = 0. 
\end{equation}
Again, the term  `$  \lim_ {\lambda \to 1^+} $' in \eqref{1.8} can be equivalently replaced by `$ \inf_  {\lambda >1} $'. 

Analogously to \eqref{1.7}, condition \eqref{1.8} is satisfied if and only if for every  $\e>0$ there exist $x_0=x_0(\e)>1 $ and $ \lambda =\lambda(\e) >1$, the latter one is as close as to 1 as we  want, such that 
\begin{equation} \label{1.9}
|s(t) - s(x) | \le \e \qquad \text{whenever} \quad x_0\le x < t \le x^\lambda . 
\end{equation}
It is easy to see that a real-valued function $s$ is slowly oscillating if and only if $s$ is both slowly decreasing and slowly increasing. 

Historically,  the term `slow oscillation' was introduced by Hardy \cite{harom} (see also in \cite[p.~124]{negy}), in the case of the summability $(C,1)$ of sequences of numbers. 

We note that in the special case when 
\[
s(x) := \int _1^x f(u) \,du, \qquad x\ge1, 
\]
where $f\in L^1_\loc[1,\infty)$, one can easily get sufficient conditions for the fulfillment of \eqref{1.7} and \eqref{1.9}, respectively. If $f$ is a real-valued function such that 
\begin{equation}\label{1.10}
u (\log u) f(u) \ge -C \qquad \text{at almost every} \quad u> x_0 ,
\end{equation}
where $C>0$ and $x_0\ge 1$ are constants, then the function $s$ is slowly decreasing with respect to summability $(L,1)$. Furthermore, if $f$ is a complex-valued function such that 
\begin{equation}\label{1.11}
 (\log u) |f(u)| \le C \qquad \text{at almost every} \quad u> x_0 ,
\end{equation}
where $C>0$ and $x_0\ge 1$ are constants, then the function $s$ is slowly oscillating  with respect to summability $(L,1)$.

Condition \eqref{1.10} is called a one-sided Tauberian condition, while \eqref{1.11} is called a two-sided Tauberian condition. These terms go back to Landau \cite{nyolc} with respect to summability  $(C,1)$ of sequences of real numbers; and to Hardy  \cite{harom} (see also in \cite[p.~124]{negy}) with respect to summability $(C,1)$ of sequences of real or complex numbers.  

The following two classical Tauberian theorems were proved in \cite[Corollaries 1 and 2]{ujtiz}.

\begin{thma}\label{tha} 
If a function\/ $s:[1,\infty) \to \R $ is slowly decreasing with respect to summability\/ $(L,1)$, then the implication\/ $\eqref{1.5} \imply \eqref{1.3} $ holds true.
\end{thma}

\begin{thmb}\label{thb} 
If a function\/ $s:[1,\infty) \to \C $ is slowly oscillating with respect to summability\/ $(L,1)$, then the implication\/ $\eqref{1.5} \imply \eqref{1.3} $ holds true.
\end{thmb}

We note that in the case of sequences of real numbers, a theorem analogous to Theorem~A was proved by Kwee \cite[Lemma~3]{het}.  

\section{New results} 

First, we prove that if a measurable function $s$ is slowly decreasing or oscillating with respect to summability $(L,1)$, then the existence of the stastistical limit~$\ell$ of~$s$ at~$\infty$ implies the existence of the ordinary limit of $s$ at $\infty$ to the same limit $\ell$. 

\begin{thm}\label{th1} 
If a real-valued function\/ $s:[1,\infty) \to \R $ is measurable and slowly decreasing with respect to summability\/ $(L,1)$, then the implication\/ $\eqref{1.2} \imply \eqref{1.3} $ holds true.
\end{thm}

\begin{thm}\label{th2} 
If a complex-valued function\/ $s:[1,\infty) \to \C $ is measurable and slowly oscillating with respect to summability\/ $(L,1)$, then the implication\/ $\eqref{1.2} \imply \eqref{1.3} $ holds true.
\end{thm}

The next two theorems are the main results of the present paper. They state that under the Tauberian condition of slow decrease or slow oscillation, respectively,  \eqref{1.3} follows from the existence of the even weaker limit 
\begin{equation}\label{2.1}
\stlim_{t\to \infty} \tau(t) = \ell.
\end{equation} 

\begin{thm}\label{th3} 
If $s\in L^1_\loc [1,\infty)$ is a real-valued function and slowly decreasing with respect to summability\/ $(L,1)$, then the implication\/ $\eqref{2.1} \imply \eqref{1.3} $ holds true.
\end{thm}

\begin{thm}\label{th4} 
If $s\in L^1_\loc [1,\infty)$ is a complex-valued function and slowly oscillating with respect to summability\/ $(L,1)$, then the implication\/ $\eqref{2.1} \imply \eqref{1.3} $ holds true.
\end{thm}

We note that analogous theorems were proved in \cite{tiz} for sequences of real and complex numbers, respectively. However, the method of the proof in the present paper is more straightforward than that in \cite{tiz}. As a result, the present proofs are essentially shorter and more transparent than those in \cite{tiz}. 

\section {Auxiliary results} 

Our Lemma~\refz{l1} is analogous to the famous Vijayaraghavan lemma (see in \cite{tizenharom} and also in \cite[Theorem~239 on p.~307]{negy}), which relates to the slow decrease with respect to summability $(C,1)$ in the case of sequences of real numbers. Our Lemma~\ref{l1} relates to slow decrease with respect to summability $(L,1)$ in the case of real-valued functions. 

\begin{lem}\label{l1} 
If a function\/ $s:[1,\infty) \to \R $ is such that the condition\/ \eqref{1.7} is satisfied only for $\e:=1$, where $x_0>1$ and $\lambda>1$, then there exists a constant $B_1>0$ such that 
\begin{equation} \label{3.1} 
s(t) - s(x) \ge -B_1 \log \Big( \frac {\log t } { \log x} \Big)  \qquad \text{whenever} \quad x_0\le x < t^{1/\lambda}. 
\end{equation}
\end{lem} 

\begin{proof} 
Given $x_0\le x < t^{1/\lambda} $, we form the decreasing sequence 
\begin{equation} \label{3.2} 
t_0:= t, \qquad t_p := t_{p-1} ^{1/\lambda}, \qquad p=1,2,\dots, q+1,
\end{equation}
where $q$ is defined by the condition
\begin{equation} \label{3.3} 
t_{q+1}  \le x < t_q . 
\end{equation}
By \eqref{1.7} and \eqref{3.3}, we estimate as follows: 
\begin{equation} \label{3.4} 
s(t)-s(x) = \sum_{p=1} ^q \big ( s(t_{p-1}) -s(t_p) \big) + \big( s(t_q) -s(x) \big) \ge -q-1. 
\end{equation}
It is clear that 
\begin{equation} \label{3.5} 
\frac 1 {\lambda^q} \log t > \log x, \quad \text {or equivalently} \quad q< \frac 1 {\log \lambda} \log\Big( \frac{\log t}{\log x} \Big).
\end{equation}
Combining \eqref{3.4} and \eqref{3.5} gives 
\begin{equation} \label{3.6} 
s(t)-s(x) \ge -1- \frac 1 {\log \lambda} \log\Big( \frac{\log t}{\log x} \Big) \qquad \text {whenever} \quad x_0\le x < t^{1/\lambda}. 
\end{equation}
Since it follows from $ x< t^{1/\lambda} $ that 
\begin{equation} \label{3.7} 
\log \lambda < \log\Big( \frac{\log t}{\log x} \Big) , 
\end{equation}
we conclude from \eqref{3.6} that \eqref{3.1} holds with $B_1:= 2/\log \lambda$. 
\end{proof} 

Our Lemma~\ref{l2} is the counterpart of Lemma~\ref{l1} in the case of complex-valued functions. 

\begin{lem}\label{l2} 
If a function\/ $s:[1,\infty) \to \C $ is such that the condition\/ \eqref{1.9} is satisfied only for $\e:=1$, where $x_0>1$ and $\lambda>1$, then with  $B_1 := 2/ \log \lambda$  we have  
\begin{equation} \label{3.8} 
|s(t) - s(x) | \le B_1 \log \Big( \frac {\log t } { \log x} \Big)  \qquad \text{whenever} \quad x_0\le x < t^{1/\lambda}. 
\end{equation}
\end{lem} 

\begin{proof} 
It goes along the same lines as the proof of Lemma~\ref{l1}. For given  $x_0\le x < t^{1/\lambda} $, we define $t_0$, $t_1$, \dots, $t_{q+1} $ by \eqref{3.2} and \eqref{3.3}. Using \eqref{1.9} and \eqref{3.4} gives 
\begin{equation} \label{3.9} 
|s(t)-s(x)| \le \sum_{p=1} ^q \big | s(t_{p-1} -s(t_p) \big| + \big| s(t_q) -s(x) \big| \le q+1. 
\end{equation}
Combining \eqref{3.5} and \eqref{3.9} we obtain 
\[
|s(t)-s(x)| \le 1 + \frac 1 {\log \lambda} \log\Big( \frac{\log t}{\log x} \Big) \qquad \text {whenever} \quad x_0\le x < t^{1/\lambda} 
\]
(c.f.\ \eqref{3.6}). Tking into account \eqref{3.7}, hence \eqref{3.8} follows with the same constant $B_1 := 2/\log\lambda$ as in Lemma~\ref{l1}. 
\end{proof}

The next two lemmas are corollaries of Lemmas \ref{l1} and \ref{l2}, respectively. 

\begin{lem}\label{l3} 
Under the assumptions of Lemma\/ \refz{l1}, there exists a constant $B_2>0$ such that 
\begin{equation}\label{3.10} 
\frac 1 {\log t} \int _{x_0} ^t \frac{s(t)-s(x) }{ x} \,dx \ge -B_2 \qquad \text{whenever} \quad  t> x_0 ^\lambda .
 \end{equation}
\end{lem} 

\begin{proof} 
Without loss of generality, we may assume that $x_0 >e $. By \eqref{1.7} with $\e:=1$ and \eqref{3.1}, we estimate as follows:
\begin{align}\label{3.11}
\int _{x_0} ^t &\frac{s(t)-s(x) }{ x} \,dx = \Big\{ \int _{x_0}^{t^{1/\lambda}} + \int _{t^{1/\lambda}}^ t \Big\} \frac{s(t)-s(x) }{ x} \,dx \\
&\qquad \ge -B_1 \int _{x_0}^{t^{1/\lambda}} \frac{1}{x} \log \Big( \frac {\log t}{\log x} \Big) \,dx - \int _{t^{1/\lambda}}^ t \frac{dx}{x} \nonumber \\
&\qquad \ge -B_1 (\log\log t) \int _{x_0}^{t^{1/\lambda}} \frac{dx}{x} + B_1 \int _{x_0}^{t^{1/\lambda}} \frac{\log\log x}{x} \,dx -\frac{\lambda-1}{\lambda} \log t. \nonumber
\end{align} 
Integration by parts gives 
\begin{align}\label{3.12}
& \int _{x_0}^{t^{1/\lambda}}  \frac{\log\log x}{x} \,dx = \Big[ (\log\log x) \log x \Big ] _{x_0}^{t^{1/\lambda}} - \int _{x_0}^{t^{1/\lambda}}  \frac{dx}{x} \\
&\qquad = (\log \log (t^{1/\lambda}) \log t^{1/\lambda} - (\log \log x_0) \log x_0 - \log t^{1/\lambda} + \log x_0 \nonumber\\
&\qquad = \frac { (\log \log t) \log t }{\lambda} - \frac{\log \lambda}{\lambda} \log t - (\log \log x_0) \log x_0 -\log  t^{1/\lambda} + \log x_0. \nonumber
\end{align} 
Returning to \eqref{3.11}, we obtain 
\begin{align}\label{3.13} 
& \int _{x_0}^t \frac{s(t)-s(x)}{x} \,dx \\
&\qquad \ge -B_1 \frac{\log \lambda}{\lambda} \log t  - B_1 (\log \log x_0) \log x_0 - B_1 \frac{\log t}{\lambda} +B_1 \log x_0  \nonumber\\
&\qquad \ge -B_1 (\log t)\Big(  \frac{\log \lambda}{\lambda} + \frac{ (\log \log x_0) \log x_0 }{\log t} + \frac 1\lambda \Big)  \nonumber\\
&\qquad \ge -B_1 (\log t) \Big(  \frac{\log \lambda}{\lambda} + \frac{ (\log \log x_0)}{\lambda } + \frac 1\lambda \Big) ,  \nonumber
\end{align} 
where we took into account that $(\log x_0) / (\log t ) < 1/\lambda$. Now, the last expression on the right-hand side of \eqref{3.13} proves \eqref {3.11} with the constant 
\begin{equation}\label {3.14} 
B_2 := \frac {B_1}{\lambda} ( \log \lambda + \log \log x_0 +1 ) . 
\end{equation}
\end{proof}

The counterpart of Lemma~\ref{l3} in the case of complex-valued functions reads as follows. 

\begin{lem}\label{l4} 
Under the assumptions of Lemma\/ \refz{l2}, there exists a constant $B_2>0$ such that 
\begin{equation}\label{3.15} 
\frac 1 {\log t} \int _{x_0} ^t \frac{ |s(t)-s(x)| }{ x} \,dx \le B_2 \qquad \text{whenever} \quad  t> x_0 ^\lambda .
\end{equation}
\end{lem} 

\begin{proof} 
It goes along analogous lines as the proof of Lemma~\ref{l3}. By \eqref{1.9}  with $\e:=1$ and \eqref{3.8}, we estimate as follows.
\begin{equation} \label{3.16}
\int _{x_0} ^t \frac{|s(t)-s(x) |}{ x} \,dx \le B_1 \int _{x_0}^{t^{1/\lambda}} \frac{1}{x} \log \Big( \frac {\log t}{\log x} \Big) \,dx + 
\int _{t^{1/\lambda}}^ t \frac{dx}{x} 
\end{equation}
(cf.~\eqref{3.11}). Combining  \eqref{3.12} and {3.16} we obtain 
\begin{equation}
\int _{x_0} ^t \frac{|s(t)-s(x) |}{ x} \,dx \le  B_1 (\log t) \Big(  \frac{\log \lambda}{\lambda} + \frac{ (\log \log x_0)}{\lambda } + \frac 1\lambda \Big) . 
\end{equation}
Thus, we have proved \eqref{3.15} with the same constant $B_2$ as given in \eqref {3.14}. 
\end{proof}

\section{Proofs of Theorems \ref{th1}--\ref{th4}}

\begin{proof}[Proof of Theorem~\refz{th1}] 
Let $\e>0$, $x_0>1$ and $\lambda >1$ be arbitrarily given. By assumption, the statistical limit $\ell$ of the function $s$ exists at $\infty$. Thus, by \eqref{1.1} with $a:=1$, there exists $b_1\ge x_0$ such that 
\[ 
|s(b_1)-\ell| \le \e.
\]
We distinguish between two cases: 

\begin{enumerate} 
\item[(i)] There exists some $b_2 \in ( b_1^{ \sqrt\lambda} , b_1 ^\lambda)$ such that 
\begin{equation}\label{4.1} 
|s(b_2) - \ell| \le \e; 
\end{equation} 
\item[(ii)] There is no such $b_2$; that is, we have 
\begin{equation*}\
|s(t) - \ell| > \e\qquad \text{for every}  \qquad t \in ( b_1^{ \sqrt\lambda} , b_1 ^\lambda).
\end{equation*} 
\end{enumerate}

In the latter case, we choose some $b_2 \ge b_1 ^\lambda$ for which \eqref{4.1} is satisfied. Due to \eqref{1.1}, such $b_2 $ certainly exists. 

We repeat the previous step by starting with $b_2$, and so on. As a result, we obtain an increasing sequence $(b_n : n=1,2,\dots)$ of numbers such that 
\begin{equation}\label{4.2} 
|s(b_n) - \ell| \le \e, \qquad n=1,2,\dots . 
\end{equation} 
We claim that the case when 
\begin{equation}\label{4.3} 
|s(t) - \ell| > \e, \qquad \text{for every} \quad t \in ( b_n^{ \sqrt\lambda} , b_n ^\lambda)
\end{equation} 
cannot occur for infinitely many values of $n$. Otherwise, for infinitely many $n$ we would have 
\[
\frac 1{b_n} \Big| \{ t\in(1,b_n) : |s(t)-\ell| >\e \} \Big| 
=  b_n ^{\lambda-1} -b_n^{ \sqrt\lambda -1} > b_1 ^{\lambda-1} -b_1^{ \sqrt\lambda -1} >0,  
\]
and this contradicts \eqref{1.1}. Consequently, inequality \eqref{4.3} can occur only for finitely many values of $n$. Denote by $n_0$ the largest value of $n$ for which \eqref{4.3} holds (perhaps $n_0=0$ in the case when \eqref{4.3} does not occur at all). Consequently, we have 
\begin{equation}\label{4.4} 
b_{n+1} < b_n ^\lambda \qquad \text{for} \quad n=n_0+1, n_0+2,\dots\,. 
\end{equation} 
On the other hand, by definition we also have 
\begin{equation*} 
b_{n+1} > b_n ^{\sqrt\lambda} \qquad \text{for} \quad n=n_0+1, n_0+2,\dots\,, 
\end{equation*} 
whence it follows that 
\[
\lim_{n\to\infty} b_n =\infty. 
\]

By the condition \eqref{1.7} of slow decrease, we have 
\begin{equation} \label{4.5} 
s(t) - s(b_n) \ge -\e \qquad \text{whenever} \quad x_0\le b_n<t\le b_n^{\lambda} , \ \ n>n_0. 
\end{equation} 
Now, let $t\in (b_n, b_{n+1} ]$ for some $n>n_0$. By \eqref{4.4}, we have 
\begin{equation} \label{4.6} 
b_n< t \le b_{n+1} < b_n^{\lambda} <t^\lambda . 
\end{equation} 

On the one hand, it follows from \eqref{4.2} and \eqref{4.5} that if $n>n_0$, then for every $t\in (b_n, b_{n+1} ]$,  
\begin{equation} \label{4.7} 
s(t)-\ell = (s(t)-s(b_n)) + (s(b_n)-\ell) \ge -2\e. 
\end{equation} 
On the other hand, it follows from \eqref{4.2} and \eqref{4.4}--\eqref{4.6} that 
\begin{equation} \label{4.8} 
s(t)-\ell = (s(t)-s(b_{n+1})) + (s(b_{n+1})-\ell) \le 2\e. 
\end{equation} 

Putting together \eqref{4.7} and \eqref{4.8} gives 
\[
 | s(t)-\ell| \le 2\e \qquad \text{for every} \quad t \in \bigcup _{n=n_0+1}^\infty (b_n, b_{n+1}] = (b_{n_0+1} , \infty). 
\]
Since $\e>0$ is arbitrary, this proves that the ordinary limit of $s$ exists at $\infty $ and it equals $\ell$. 
\end{proof}

\begin{proof}[Proof of Theorem~\refz{th2}] 
It is analogous to the proof of Theorem~\ref{th1}. Again, we can show that for every $\e>0$ and $\lambda>1$, there exists an increasing sequence $(b_n : n=1,2,\dots)$ of numbers tending to $\infty$,  while conditions \eqref{4.2} and \eqref{4.4} are also satisfied. 

By the condition \eqref{1.9} of slow oscillation, we have 
\begin{equation} \label{4.9} 
 | s(t) - s(b_n) | \le \e \qquad \text{whenever} \quad x_0\le b_n<t\le b_n^{\lambda} 
\end{equation}
(cf.~\eqref{4.5}). Now, it follows from \eqref{4.2}, \eqref{4.4} and \eqref{4.9} that 
\begin {align*} 
 | s(t)-\ell| &\le | s(t) - s(b_n)| + | s(b_n) -\ell| \le 2\e \\
 & \qquad \qquad \text{for every} \quad t \in \bigcup _{n=n_0+1}^\infty (b_n, b_{n+1}] = (b_{n_0+1} , \infty). 
\end{align*} 
Since $\e>0$ is arbitrary, this proves that the ordinary limit of $s$ exists at $\infty $ and it equals $\ell$. 
\end{proof}

\begin{proof}[Proof of Theorem~\refz{th3}] 
It hinges on Lemma~\ref{l3} and Theorem~\ref{th1}.

First, we prove that if the condition \eqref{1.7} of slow decrease is satisfied for a single $\e>0$, say $\e:=1$, then we have 
\begin{equation} \label{4.10} 
\liminf _{x\to\infty} \frac{s(x)}{x} \ge 0. 
\end{equation}
Indeed, from \eqref{1.7} with $\e=1$ it follows that for $p=1,2,\dots$ we have 
\[
s(x_0^{\lambda^p} )- s(x_0) \ge -p , \qquad \text{where} \quad x_0:=x_0(1)>1 \text{ and } \lambda=\lambda(1) >1; 
\]
whence we conclude that 
\[
\frac {s(x_0^{\lambda^p} ) } {x_0^{\lambda^p} } \ge \frac { s(x_0) } {x_0^{\lambda^p} } - \frac {p} { x_0^{\lambda^p} } \to 0 \qquad \text{as} \quad p\to \infty. 
\]
Now, the fulfillment of \eqref{4.10} is obvious. 

Second, we prove that if the real-valued function $s\in L^1_\loc [1,\infty)$ is slowly decreasing, then so is its logarithmic mean $\tau(t)$ defined in \eqref{1.4}. To this effect, let some $0<\e<1$ be given, and let $x_0 \le x <t\le x^\lambda$, where $x_0=x_0(\e)$ and $\lambda = \lambda(\e) $ occur in \eqref{1.7} and this time $\lambda$ is chosen so close to $1$ that 
\begin{equation}\label{4.11} 
1 < \lambda \le 1+ \frac{\e} { \max\{1, B_2\} } \,, 
\end{equation} 
where $B_2$ is from \eqref{3.10}. 

By definiton \eqref{1.4}, we estimate as follows
\begin{align} \label{4.12} 
 &\tau(t) - \tau(x) := \frac 1{\log t} \int _1^t \frac {s(u)} {u} \,du  - \frac 1{\log x} \int _1^x \frac {s(u)} {u} \,du \\
&\qquad = - \Big( \frac1{\log x} -\frac 1{\log t } \Big) \int _1^x \frac {s(u)} {u} \,du  + \frac 1{\log t } \int _x^t \frac {s(u)} {u} \,du \nonumber \\
&\qquad = \Big( \frac1{\log x} -\frac 1{\log t } \Big) \int _1^x  \frac {s(x)- s(u)} {u} \,du   + \frac 1{\log t } \int _x^t \frac {s(u)-s(x)} {u} \,du \nonumber \\
&\qquad = \Big( \frac1{\log x} -\frac 1{\log t } \Big) \Big\{ \int_1^{x_0} + \int_{x_0}^x \Big\}  \frac {s(x)- s(u)} {u} \,du + \frac 1{\log t } \int _x^t \frac {s(u)-s(x)} {u} \,du\nonumber  \\
&\qquad = \Big( \frac1{\log x} -\frac 1{\log t } \Big) (\log x_0) s(x) - \Big( \frac1{\log x} -\frac 1{\log t } \Big) \int_1^{x_0} \frac {s(u)} {u} \,du \nonumber \\
&\qquad \qquad + \Big( \frac1{\log x} -\frac 1{\log t } \Big) \int _{x_0} ^x  \frac {s(x)- s(u)} {u} \,du +  \frac 1{\log t } \int _x^t \frac {s(u)-s(x)} {u} \,du \nonumber \\
&\qquad =: J_1 + J_2 + J_3 +J_4, \qquad \text{say}. \nonumber 
\end{align} 

It follows from \eqref{4.10} that 
\begin{equation} \label{4.13} 
\liminf _{x\to\infty} J_1 \ge 0. 
\end{equation}

Since $s\in L^1_\loc [1,\infty)$, we have 
\begin{equation} \label{4.14} 
\lim _{x\to\infty} J_2 = 0. 
\end{equation}

It follows from $x<t\le x^\lambda $ that 
\begin{equation} \label{4.15} 
\frac 1 \lambda \le \frac {\log x }{\log t } \, . 
\end{equation}
Using inequalities \eqref{4.11} and \eqref {4.13}, by Lemma~\ref{l3} we estimate as follows
\begin{equation} \label{4.16} 
J_3 \ge -B_2 \Big ( 1 -  \frac {\log x }{\log t } \Big) \ge -B_2 \Big(1-\frac 1{\lambda} \Big) \ge -B_2 (\lambda-1) = -\e. 
\end{equation}

Using again the condition \eqref{1.7} of slow decrease; \eqref{4.11} and \eqref{4.15} gives 
\begin{equation} \label{4.17} 
J_4 \ge - \frac 1 {\log t} \int _ x ^t \frac {du}{u} =  - \Big ( 1 -  \frac {\log x }{\log t } \Big) \ge - \Big(1-\frac 1{\lambda} \Big)  \ge -(\lambda-1) \ge -\e. 
\end{equation}

Combining \eqref{4.12}--\eqref{4.17} yields 
\[
\tau(t) - \tau(x) \ge -4\e \qquad \text {whenever} \quad x_0\le x < t \le x^\lambda, 
\]
provided that $x$ is large enough, where we also took into account the limit relations in \eqref{4.13} and \eqref{4.14}. This proves that $\tau(t)$ is also slowly decreasing. Therefore, we may apply Theorem~\ref{th1} in Section 2, according to which $\tau(t)$ converges in the ordinary sense as $t\to \infty$ to the same limit $\ell$. 

Finally, we apply Theorem~A in Section 1 to conclude the ordinary convergence of $s(x)$ to $\ell$ as $x\to \infty$, again due to the slow decrease of the function $s$. 

The proof of Theorem~\ref{th3} is complete. 
\end {proof}

\begin{proof}[Proof of Theorem~\refz{th4}] 
It is analogous to the proof of Theorem~\ref{th3}, while using Lemma~\ref{l4} instead of Lemma~\ref{l3}, and applying Theorem~B instead of Theorem~A in the last step in the proof. The details are left to the reader. 
\end{proof}

\end{document}